\documentclass[12pt,a4paper]{amsart}
\usepackage[nocompress]{cite}
\usepackage{listings}
\usepackage{docmute}
\usepackage{amssymb, enumerate, calligra}
\usepackage{amsthm}
\usepackage{amsmath}
\usepackage{amsxtra}
\usepackage{latexsym}
\usepackage{mathrsfs}
\usepackage[all,cmtip]{xy}
\usepackage[all]{xy}
\usepackage{enumitem}
\usepackage{xcolor}
\usepackage{graphicx}
\usepackage{comment}
\usepackage{mathabx,epsfig}
\def\acts{\ \rotatebox[origin=c]{-90}{$\circlearrowright$}\ }
\def\racts{\ \rotatebox[origin=c]{90}{$\circlearrowleft$}\ }

\usepackage[dvipdfmx]{hyperref}

\newtheorem{thm}{Theorem}[section]
\newtheorem{lem}[thm]{Lemma}
\newtheorem{conj}[thm]{Conjecture}
\newtheorem{claim}[thm]{Claim}
\newtheorem{prop}[thm]{Proposition}
\newtheorem{cor}[thm]{Corollary}
\theoremstyle{definition}
\newtheorem{defn}[thm]{Definition}
\newtheorem{rmk}[thm]{Remark}
\newtheorem{ex}[thm]{Example}
\newtheorem*{ack}{Acknowledgements}
\numberwithin{equation}{section}

\def\C{{\mathbb C}}
\def\Q{{\mathbb Q}}
\def\R{{\mathbb R}}
\def\Z{{\mathbb Z}}
\def\P{{\mathbb P}}

\def\QQ{\overline{\mathbb Q}}

\DeclareMathOperator{\pr}{pr}
\DeclareMathOperator{\Pic}{Pic}

\DeclareMathOperator{\id}{id}
\DeclareMathOperator{\Aut}{Aut}

\DeclareMathOperator{\End}{End}

\DeclareMathOperator{\Sym}{Sym}
\DeclareMathOperator{\Ker}{Ker}

\DeclareMathOperator{\Nef}{Nef}
\DeclareMathOperator{\Int}{Int}
\DeclareMathOperator{\GL}{GL}
\DeclareMathOperator{\Eff}{Eff}
\DeclareMathOperator{\ord}{ord}

\newcommand{\relmiddle}[1]{\mathrel{}\middle#1\mathrel{}}

\newenvironment{claimproof}[0]
  {%
   \paragraph{\it Proof.}%
  }
  {%
    \hfill$\blacksquare$%
  }

\newenvironment{notation}[0]{%
  \begin{list}%
    {}%
    {\setlength{\itemindent}{0pt}
     \setlength{\labelwidth}{4\parindent}
     \setlength{\labelsep}{\parindent}
     \setlength{\leftmargin}{5\parindent}
     \setlength{\itemsep}{0pt}
     }%
   }%
  {\end{list}}

\title[]
{Kawaguchi-Silverman conjecture for endomorphisms on several classes of varieties}
\author{Yohsuke Matsuzawa}
\address{Department of Mathematics, box 1917, Brown University, Providence, Rhode Island 02912, USA}
\email{matsuzawa@math.brown.edu}

\begin{document}

\begin{abstract}
We prove Kawaguchi-Silverman conjecture (KSC) and Shibata's conjecture on ample canonical heights for endomorphisms on several classes of algebraic varieties
including varieties of Fano type and projective toric varieties.
We also prove KSC for group endomorphisms of linear algebraic groups.
We also propose a possible approach to the conjecture using equivariant MMP.
\end{abstract}

\maketitle

\setcounter{tocdepth}{1}
\tableofcontents

\section{Introduction}

Let $X$ be a smooth projective variety and $f \colon X \dashrightarrow X$ a dominant rational self-map,
both defined over $ \overline{\mathbb Q}$.
Silverman introduced the notion of arithmetic degree in \cite{sil},
which measures the arithmetic complexity of $f$-orbits by means of Weil height functions.

On  the other hand, we can attach the dynamical degree $\delta_{f}$ to $f$, which  measures the geometric complexity of
the dynamical system.
In \cite{sil}, \cite[Conjecture 6]{ks3} Kawaguchi and Silverman conjectured that the arithmetic degree of any Zariski dense orbits
are equal to the first dynamical degree $\delta_{f}$ (cf.  Conjecture \ref{ksc}). 

In \cite{shi}, Shibata defines the so called ample canonical height function, which is a variant of dynamical canonical height,
and proposes a conjecture on the vanishing loci of ample canonical heights.
This conjecture implies Kawaguchi-Silverman conjecture for endomorphism of (smooth) projective varieties.

In this paper, we prove the Kawaguchi-Silverman conjecture and Shibata's conjecture for endomorphisms on several classes of algebraic varieties as well as
propose a possible approach to the conjecture using equivariant MMP.
The contents of this paper are as follows.

In \S\ref{sec:def}, we recall the definition of arithmetic degrees and introduce Kawaguchi-Silverman conjecture and Shibata's conjecture.

In \S\ref{sec:pre}, we collect preliminary facts with proofs.

\S\ref{sec:shi} proves Shibata's conjecture (therefore Kawaguchi-Silverman conjecture) for endomorphisms of normal projective varieties $X$ 
such that the nef cone is generated by finitely many semi-ample divisors and $\Pic X_{\Q} =N^{1}(X)_{\Q}$.
This class of varieties includes,  by definition, all Mori dream spaces.
In particular, the conjecture holds for endomorphisms of varieties of Fano type and projective toric varieties.

In our earlier paper \cite{mss}, we prove Kawaguchi-Silverman conjecture for non-isomorphic endomorphisms of smooth projective surfaces.
One of the key ingredients in that proof is that minimal model program (MMP) works equivariantly with respect to non-isomorphic endomorphisms.
We need not just non-isomorphic but {\it int-amplified} to do the same in higher dimension.
In \S\ref{sec:equiMMP}, we prove Kawaguchi-Silverman conjecture for some special classes of endomorphisms by using 
equivariant MMP established by Meng and Zhang.
We also propose a strategy to prove the conjecture for endomorphisms on varieties admitting int-amplified endomorphism by induction on dimension. 
We prove the following purely geometric fact, which might be a key to proceed the induction:
Let $X$ be a $\Q$-factorial klt rationally connected variety admitting an int-amplified endomorphism.
Then, every nef Cartier divisor on $X$ is $\Q$-linearly equivalent to an effective divisor.

In \S\ref{sec:grp}, we prove (part of) Kawaguchi-Silverman conjecture for group endomorphisms of linear algebraic groups.
As it turns out in the proof, the essential case is the torus case, which is proved by Silverman in \cite{sil}.

\vspace{15pt}
{\bf Notation.}
\begin{itemize}
\item Throughout this paper, the ground field is $\QQ$ unless otherwise stated.
All statements that are purely geometric hold over any algebraically closed field of characteristic zero.
\item A variety over a field $k$ is a geometrically integral separated scheme of finite type over $k$.
A divisor on a variety means a Cartier divisor.
\item Let $X$ be a projective variety over an algebraically closed field of characteristic zero.
\begin{itemize}
\item $N^{1}(X)$ is the group of Cartier divisors modulo numerical equivalence
(a Cartier divisor $D$ is numerically equivalent to zero, which is denoted by $D\equiv0$, if $(D\cdot C)=0$ for all irreducible curves $C$ on $X$).
\item $N_{1}(X)$ is the group of $1$-cycles modulo numerical equivalence
(a $1$-cycle $ \alpha$ is numerically zero if $(D\cdot \alpha)=0$ for all Cartier divisors $D$).
By definition, $N^{1}(X)$ and $N_{1}(X)$ are dual to each other.
\item When $X$ is normal, the Iitaka dimension of a $\Q$-Cartier divisor $D$ on $X$ is denoted by $\kappa(D)$. 
\end{itemize}
\item Let $M$ be a $\Z$-module. We write $M_{\Q}=M {\otimes}_{\Z}\Q$, $M_{\R}=M {\otimes}_{\Z}\R$, and so on.
\end{itemize}

\begin{ack}
The author would like to thank Tomohide Terasoma for attending his seminar and also supporting his study.
The author would like to thank Takeru Fukuoka for a lot of insightful conversations.
He also would like to thank Kaoru Sano and Takahiro Shibata for discussing with me this topic.
The author would like to thank Sho Ejiri, Kenta Hashizume, Akihiro Kanemitsu, Masaru Nagaoka, Sho Tanimoto, and Shou Yoshikawa
for answering his questions.
Part of this work was done during my stay at Brown University and the University of Michigan.
The author would like to thank Joe Silverman and Mattias Jonsson for their hospitality. 
The author is supported by JSPS Research Fellowship for Young Scientists and KAKENHI Grant Number 18J11260
and the Program for Leading Graduate Schools, MEXT, Japan.
\end{ack}

\section{Arithmetic degree and some Conjectures}\label{sec:def}

\subsection{Dynamical degrees}

Let $X$ be a smooth projective variety defined over an algebraically closed field of characteristic zero
and $f \colon X \dashrightarrow X$ a  dominant rational map.
We define pull-back $f^{*} \colon N^{1}(X) \longrightarrow N^{1}(X)$ as follows.
Take a resolution of indeterminacy $\pi \colon X' \longrightarrow X$ of $f$ with $X'$ smooth projective.
Write $f'=f \circ \pi $.
Then we define $f^{*}=\pi_{*}\circ {f'}^{*}$.
This is independent of the choice of resolution.

\begin{defn}

Fix a norm $||\ ||$ on the finite dimensional real vector space $\End(N^{1}(X)_{\R})$.
Then the (first) dynamical degree of $f$ is 
\[
\delta_{f}=\lim_{n \to \infty}||(f^{n})||^{1/n}.
\]
This is independent of the choice of $||\ ||$.
We refer to \cite{dang, df, tru0}, \cite[\S 4]{ds} for basic properties of dynamical degrees.
\end{defn}

\begin{rmk}
The dynamical degree has another equivalent definition in terms of intersection numbers:
\[
\delta_{f}=\lim_{n \to \infty} ((f^{n})^{*}H\cdot H^{\dim X-1})^{1/n}
\]
where $H$ is any nef and big Cartier divisor on $X$ (cf. \cite{dang}).
\end{rmk}

\begin{rmk}\label{rmk on dyn deg}\ 
Dynamical degree is a birational invariant. That is, if $\pi \colon X \dashrightarrow X'$ is a birational map
and $f \colon X \dashrightarrow X$ and $f' \colon X' \dashrightarrow X'$ are conjugate by $\pi$, then $\delta_{f}=\delta_{f'}$.
In particular, we can define the dynamical degree of a dominant rational self-map of quasi-projective varieties by taking
a smooth projective model.
\end{rmk}

\subsection{Arithmetic degrees}

In this subsection, the ground field is $ \overline{\mathbb Q}$.
We briefly recall the definition of Weil height function.
Standard references for Weil height functions are \cite{bg,hs,Lan}, for example.

The height function on a projective space $\P^{N}(\QQ)$ is
\begin{align*}
\P^{N}(\QQ) \longrightarrow \R\ ;\  (x_{0}:\cdots:x_{N}) \mapsto \frac{1}{[K : \Q]}\sum_{v} \log \max\{|x_{0}|_{v}, \dots , |x_{N}|_{v}\}
\end{align*}
where $K$ is a number field (finite extension of $\Q$ contained in the fixed algebraic closure $\QQ$) containing the coordinates $x_{0},\dots,x_{N}$,
the sum runs over all places $v$ of $K$, and $|\ |_{v}$ is the absolute value associated with $v$ normalized as follows:
\begin{align*}
|x|_{v}= 
\begin{cases}
\#\ ( \mathcal{O}_{K}/  \mathfrak{p}_{v})^{-\ord_{v}(x)} 
\ \text{if $v$ is non-archimedean}\\
| \sigma_{v}(x)|^{[K_{v}:\R]}\ \text{if $v$ archimedean.}
\end{cases}
\end{align*}
Here $ \mathcal{O}_{K}$ is the ring of integers of $K$.
When $v$ is non-archimedean, $ \mathfrak{p}_{v}$ is the maximal ideal corresponding to $v$ and $\ord_{v}$ is the valuation associated with $v$.
When $v$ is archimedean, $ \sigma_{v}$ is the embedding of $K$ into $\C$ corresponding to $v$.
This definition is independent of the choice of homogeneous coordinates and the number field $K$.

Let $X$ be a projective variety over $\QQ$.
A Cartier $\R$-divisor  $D$ on $X$ determines a (logarithmic) Weil height function $h_{D}$ up to bounded functions as follows.
When $D$ is a very ample integral divisor, $h_{D}$ is the composite of the embedding by $|D|$ and the height on the projective space we have just defined.
For a general $D$, we write 

\begin{align}
\label{very ample sum}
D=\sum_{i=1}^{m}a_{i}H_{i} 
\end{align}
where $a_{i}$ are real numbers and $H_{i}$ are very ample divisors.
Then we define
\[
h_{D}=\sum_{i=1}^{m}a_{i}h_{H_{i}}.
\]
The function $h_{D}$ does not depend on the choice of the representation (\ref{very ample sum})
up to bounded function.
We call any representative of the class $h_{D}$ modulo bounded functions a height function
associated with $D$.

\begin{defn}
Let $X$ be a normal projective variety or smooth quasi-projective variety defined over $\QQ$.
Let $f \colon X \longrightarrow X$ be a surjective self-morphism. 

\begin{itemize}
\item When $X$ is projective, let $H$ be an ample divisor on $X$. Fix a Weil height function $h_{H}$ associated with $H$.
\item When $X$ is a smooth quasi-projective variety, fix a smooth projective variety $ \overline{X}$ and an open embedding $X \subset \overline{X}$.
Let $H$ be an ample divisor on $ \overline{X}$ and take a Weil height function $h_{H}$ associated with $H$.
\end{itemize}

The arithmetic degree $ \alpha_{f}(x)$ of $f$ at $x \in X(\QQ)$ is defined by
\[
\alpha_{f}(x)=\lim_{n \to \infty}\max\{1, h_{H}(f^{n}(x))\}^{1/n}
\]
if the limit exists.
Since the convergence of this limit is not proved in general, we introduce the following:
\begin{align*}
&\overline{\alpha}_{f}(x)=\limsup_{n \to \infty}\max\{1, h_{H}(f^{n}(x))\}^{1/n},\\
& \underline{\alpha}_{f}(x)=\liminf_{n \to \infty}\max\{1, h_{H}(f^{n}(x))\}^{1/n}.
\end{align*}
We call $ \overline{\alpha}_{f}(x)$ the upper arithmetic degree and $ \underline{\alpha}_{f}(x)$ the lower arithmetic degree.
The definitions of the (upper, lower) arithmetic degrees are independent of the choice of $H$ and $h_{H}$ when $X$ is normal projective and
also independent of $\overline{X}$ when $X$ is smooth quasi-projective (\cite[Proposition 12]{ks3} \cite[Theorem 3.4]{mss}). 

\end{defn}
 

In \cite{ks3}, Kawaguchi and Silverman formulated the following conjecture.

\begin{conj}[KSC]\label{ksc}
Let $X$ be a normal projective variety or  smooth quasi-projective variety and $f \colon X \longrightarrow X$ a surjective morphism,
both defined over $ \overline{\mathbb Q}$.
Let $x \in X( \overline{\mathbb Q})$.
\begin{enumerate}
\item\label{ksc1} The limit defining $\alpha_{f}(x)$ exists.
\item\label{ksc2} If the orbit $ O_{f}(x)=\{ f^{n}(x) \mid n=0,1,2, \dots \}$ is Zariski dense in $X$,
then $\alpha_{f}(x)=\delta_{f}$.
\end{enumerate}
\end{conj}

\begin{rmk}
In \cite{ks3}, the conjecture is actually formulated for dominant rational self-maps of  smooth projective varieties.
\end{rmk}

\begin{rmk}
If $X$ is smooth projective and $f$ is a dominant rational map or 
$X$ is normal projective variety and $f$ is a morphism, then
the arithmetic degrees are bounded by the dynamical degree:
\[
\overline{\alpha}_{f}(x)\leq \delta_{f}
\]
for every $x\in X(\QQ)$ such that $f^{n}(x)$ is not contained in the indeterminacy locus of $f$ for all $n\geq0$ (\cite{ks3}, \cite{ma}).
Every surjective self-morphism $f$ of a smooth quasi-projective variety $X$ can be regarded 
as a dominant rational self-map of a smooth projectivization of $X$ and the same inequality holds.
\end{rmk}

\begin{rmk}\label{rmk:exist-of-ad}
If $X$ is normal projective and $f$ is a morphism, 
then $ \alpha_{f}(x)$ exists for every $x\in X(\QQ)$ and is equal to the absolute value of one of the eigenvalues of $f^{*} \colon N^{1}(X) \longrightarrow N^{1}(X)$ (\cite{ks1}).
\end{rmk}

\begin{rmk}
Let $f \colon X \longrightarrow X$ be as in Conjecture \ref{ksc} and let $x \in X(\QQ)$.
Then, $ \overline{\alpha}_{f^{n}}(x)= \overline{\alpha}_{f}(x)^{n}$ for any positive integer $n$ (\cite[Corollary 3.4]{ma}).
On the other hand, $\delta_{f^{n}}=\delta_{f}^{n}$ by definition.
Thus, if we know the existence of arithmetic degrees, we may replace $f$ by its power to prove Conjecture \ref{ksc} (\ref{ksc2}).
(Note that the $f$-orbit of $x$ is Zariski dense in $X$ if and only if the $f^{n}$-orbit of $x$ is Zariski dense in $X$.)
\end{rmk}

\begin{rmk}
Conjecture \ref{ksc} is verified in several cases.
See \cite[Remark 1.8]{mss} and also \cite{ms, ls}.
\end{rmk}

In \cite{shi}, Shibata introduces the so called ample canonical height function, which is a variant of dynamical canonical height functions.

\begin{defn}[Ample canonical height (cf. \cite{shi})]\label{def:ach}
Let $X$ be a geometrically normal projective variety defined over a number field $k$.
Let $f \colon X \longrightarrow X$ be a surjective $k$-morphism with $\delta_{f}>1$.
(The dynamical degree is defined by passing to the base change to $ \overline{k}$.)
Fix a height function $h_{X}$ associated with an ample divisor on $X$.
We take $h_{X}$ so that $h_{X}\geq1$.
Let $l_{f}$ be the minimal non-negative integer such that the sequence 
$\{h_{X}(f^{n}(x))/n^{l_{f}}\delta_{f}^{n}\}_{n=1}^{\infty}$ is bounded above for all $x\in X(\QQ)$.
(The existence of such $l_{f}$ is proved in \cite{ma}, for example).
We define the (lower) ample canonical height as
\[
\underline{h}_{f}(x)=\liminf_{n\to \infty}\frac{h_{X}(f^{n}(x))}{n^{l_{f}}\delta_{f}^{n}}.
\]
\end{defn}

In \cite{shi}, Shibata conjectures the following and proves it for several classes of self-morphisms.
(He actually treats only smooth projective varieties.)

\begin{conj}\label{shibata}
We use the notation in Definition \ref{def:ach}.
Let $K\supset k$ be a number field (contained in $ \overline{k}$).
Then, the set $Z_{f}(K)=\{x\in X(K)\mid \underline{h}_{f}(x)=0\}$ is 
contained in an $f$-invariant proper closed subset of $X_{ \overline{k}}$.
\end{conj}

\begin{rmk}
The set $Z_{f}(K)$ is independent of the choice of the ample height $h_{X}$ (\cite[\S3]{shi}). 
\end{rmk}

\begin{rmk}\label{shibata implies ksc}
Conjecture \ref{shibata} implies Conjecture \ref{ksc}(\ref{ksc2}) (\cite[\S1]{shi}).
\end{rmk}

\section{Preliminaries}\label{sec:pre}

In this section, we gather some preliminary facts that we use later.

\begin{lem}\label{lin eq eigen}
Let $X$ be a normal projective variety and $f \colon X \longrightarrow X$ a surjective morphism both defined over an algebraically closed field of characteristic zero.
The dynamical degree $\delta_{f}$ is equal to the spectral radius 
(i.e. the maximum among the absolute values of eigenvalues)
 of $f^{*} \colon N^{1}(X) \longrightarrow N^{1}(X)$.
\end{lem}
\begin{proof}
Take a resolution of singularities $\pi \colon Y \longrightarrow X$ and consider the following diagram:
\[
\xymatrix{
Y \ar[d]_{\pi} \ar@{-->}[r]^{g} & Y \ar[d]^{\pi}\\
X \ar[r]_{f} & X
}
\]
where $g$ is the induced rational map.
Note that $\delta_{f}=\delta_{g}$ by definition.
Let $\rho$ be the spectral radius of $ f^{*} \colon N^{1}(X)_{\R} \longrightarrow N^{1}(X)_{\R}$.

\begin{claim}
We have $\rho \leq \delta_{f}$.
\end{claim}
\begin{claimproof}
Since $f^{*}$ preserves the nef cone, there exists a nef $\R$-divisor $D \not \equiv 0$
such that $f^{*}D \equiv \rho D$ (cf. \cite{birkh}).
Then
\begin{align*}
(g^{n})^{*}(\pi^{*}D)\equiv (\pi \circ g^{n})^{*}D \equiv (f^{n}\circ \pi)^{*}D
\equiv \pi^{*}(f^{*})^{n}D\equiv \rho^{n}\pi^{*}D.
\end{align*}
Thus, the spectral radius of $(g^{n})^{*}$ is not less than $\rho^{n}$.
By \cite[Remark 7]{ks3}, $\delta_{g}=\lim_{n\to \infty}(\text{spectral radius of $(g^{n})^{*}$})^{1/n}$.
Thus we get $\delta_{g}\geq \rho$.

\end{claimproof}

\begin{claim}\label{dynamical deg and eigen}
We have $\delta_{f} \leq \rho$.
Therefore we have $\delta_{f}=\rho$.
\end{claim}
\begin{claimproof}
Let $H$ be an ample divisor on $X$.
Let $d=\dim X=\dim Y$.
Since $\pi^{*}H$ is nef and big, we have
\begin{align*}
\delta_{g}&=\lim_{n \to \infty}((g^{n})^{*}\pi^{*}H\cdot (\pi^{*}H)^{d-1})^{1/n}\\
&=\lim_{n \to \infty}(\pi^{*}(f^{n})^{*}H\cdot \pi^{*}H^{d-1})^{1/n}
=\lim_{n\to\infty}((f^{n})^{*}H\cdot H^{d-1})^{1/n}.
\end{align*}
Now, we define a map $||\ || \colon N^{1}(X)_{\R} \longrightarrow \R$
by
\begin{align*}
||D||=\inf \left\{ (E_{1}\cdot H^{d-1})+(E_{2}\cdot H^{d-1}) \relmiddle| \parbox[c]{.4\linewidth}{$D \equiv E_{1}-E_{2}$ where $E_{1}, E_{2}$ are effective classes} \right\}
\end{align*}
for $D \in N^{1}(X)_{\R}$.
This is a semi-norm and if $D$ is effective then $||D||=(D\cdot H^{d-1})$.
Fix a norm $||\ ||_{1}$ on $N^{1}(X)_{\R}$.
Then there exists a constant $C>0$ such that $||\ ||\leq C ||\ ||_{1}$.
Then
\begin{align*}
\delta_{f}&=\delta_{g}=\lim_{n \to \infty}||(f^{n})^{*}H||^{1/n}\\
&\leq \lim_{n \to \infty}(C||(f^{n})^{*}H||_{1})^{1/n}=\lim_{n\to \infty} ||(f^{*})^{n}H||_{1}^{1/n}\leq \rho.
\end{align*}

\end{claimproof}

\end{proof}

\begin{rmk}
By projection formula, $\delta_{f}$ is also equal to the spectral radius of $f_{*} \colon N_{1}(X) \longrightarrow N_{1}(X)$.
\end{rmk}

\begin{prop}\label{weakNorth}
Let $X$ be a normal projective variety defined over a number filed $K$.
Let $D$ be a Cartier divisor on $X_{ \overline{K}}=X\times_{K} \overline{K}$ such that $\kappa(D)>0$.
Let $h_{D}$ be a Weil height function associated with $D$ and $B, d$ positive real numbers.
Then the set
\begin{align*}
\left\{ x \in X(L) \relmiddle| \parbox[c]{.7\linewidth}{$h_{D}(x)\leq B$, $K \subset L\subset \overline{K}$ is an intermediate field with $[L:K]\leq d$}  \right\}
\end{align*} 
is not Zariski dense in $X_{ \overline{K}}$.
\end{prop}
\begin{proof}
By replacing $D$ by $mD$ for some $m>0$, we may assume the image $Y$ of the rational map $\Phi_{|D|}$ defined by the complete linear system $|D|$
has dimension $\kappa(D)>0$.
Let $\pi \colon Z \longrightarrow X_{ \overline{K}}$ be the blow up along the base ideal of  $|D|$ and
we denote the exceptional divisor by $E$.
Then $\varphi=\Phi \circ \pi \colon Z \longrightarrow X \dashrightarrow Y$ is a surjective morphism and 
there exists an ample divisor $H$ on $Y$ such that $\varphi^{*}H\sim \pi^{*}D-E$.

Now, fix height functions $h_{H}, h_{D}$, and $h_{E}$.
Then, $h_{D}\circ \pi=h_{H}\circ \varphi +h_{E}+O(1)$. 
Since $h_{E}$ is bounded below outside $E$, there exists $B'>0$ so that
\begin{align*}
&\{x\in X(L) \mid h_{D}(x)\leq B, \text{$K\subset L\subset \overline{K}$ with }[L:K]\leq d\}\\
&\subset 
\pi(E)(\QQ) \cup \pi\left(\left\{z \in Z(L) \relmiddle| \parbox[c]{.5\linewidth}{$h_{H}(\varphi(z))\leq B'$, $K\subset L\subset \overline{K}$ with $[L:K]\leq d$} \right\} \right)\\
&\subset
\pi(E)(\QQ) \cup \pi\left(\varphi^{-1}\left( \left\{ y \in Y(L) \relmiddle| \parbox[c]{.4\linewidth}{$h_{H}(y)\leq B'$, $K\subset L\subset \overline{K}$ with $[L:K]\leq d$} \right\} \right) \right).
\end{align*}
Since $H$ is ample, the set $\{y \in Y(L) \mid h_{H}(y)\leq B', \text{$K\subset L\subset \overline{K}$ with  }[L:K]\leq d\}$ is finite.
Therefore, we get the statement.

\end{proof}

\begin{prop}\label{canht-posiitaka}
Let $X$ be a normal projective variety and $f \colon X \longrightarrow X$ a surjective endomorphism, both defined over $\QQ$.
Suppose there exists a $\Q$-Cartier divisor $D$ on $X$ such that
\begin{enumerate}
\item\label{cond:eigen} $f^{*}D \sim_{\Q} dD$ and $d=\delta_{f}>1$;
\item\label{cond:iitaka} $\kappa(D)>0$.
\end{enumerate}
Then we have $ \alpha_{f}(x)=\delta_{f}$ for every $x\in X(\QQ)$ with Zariski dense $f$-orbit.
\end{prop}
\begin{proof}
Take a height function $h_{D}$ associated with $D$.
Define
\begin{align*}
\hat{h}_{D}(x)=\lim_{n\to \infty} \frac{h_{D}(f^{n}(x))}{d^{n}} \qquad (x\in X(\QQ)).
\end{align*}
Note that by condition (\ref{cond:eigen}), this limit converges and $\hat{h}_{D}=h_{D}+O(1)$.
Let $x\in X(\QQ)$ be a point with Zariski dense $f$-orbit.
If $\hat{h}_{D}(x)\leq0$, then $\hat{h}_{D}(f^{n}(x))=d^{n}\hat{h}_{D}(x)\leq0$ for all $n\geq0$.
This implies $h_{D}$ is bounded above on the orbit of $x$.
Note that $f^{n}(x)$, $n\geq0$ are defined over a same large number field.
By condition (\ref{cond:iitaka}) and Proposition \ref{weakNorth}, we get a contradiction.
Thus, we get $\hat{h}_{D}(x)>0$ and this implies $ \alpha_{f}(x)=d=\delta_{f}$.

\end{proof}

\section{Shibata's conjecture}\label{sec:shi}

\subsection{}

In this section, we prove the following theorem.

\begin{thm}\label{ksc for semi-ample}
Let $X$ be a normal projective variety and $f \colon X \longrightarrow X$ a surjective morphism with $\delta_{f}>1$, both defined over $\QQ$.
Assume that $N^{1}(X)_{\Q}\simeq \Pic (X)_{\Q}$ and the nef cone is generated by finitely many semi-ample integral divisors.
Then Conjecture \ref{shibata} holds for $f$.
In particular,  we have $ \alpha_{f}(x)=\delta_{f}$ for every point $x\in X(\QQ)$ with Zariski dense $f$-orbit.
\end{thm}

\begin{rmk}
The following classes of varieties satisfy the assumptions of the theorem and therefore Conjecture \ref{shibata} and \ref{ksc} hold for endomorphisms of them: 
\begin{enumerate}
\item Mori dream spaces;
\item varieties of Fano type (these are Mori dream space if they are $\Q$-factorial). This class contains all projective toric varieties.
\end{enumerate}
The assumptions of the theorem are actually parts of definition of Mori dream spaces.
We can directly check the assumptions for varieties of Fano type.
\end{rmk}

\begin{lem}\label{polyhedral}
Let $(X, \Delta)$ be a projective klt pair over $\QQ$ where $\Delta$ is an effective $\Q$-Weil divisor such that $-K_{X}-\Delta$ is nef and big.
Then
\begin{enumerate}
\item every nef integral divisor on $X$ is semi-ample;
\item the nef cone $\Nef(X)$ is a rational polyhedral cone;
\item $(\Pic X)_{\Q}=N^{1}(X)_{\Q}$.
\end{enumerate}
\end{lem}
\begin{proof}
(1) This follows from the base point free theorem.
(2) Take an effective divisor $E$ such that $-K_{X}-\Delta-\frac{1}{m}E$ is ample for all sufficiently large $m$ (cf. \cite[Example 2.2.19]{laz}).
Since $(X, \Delta+\frac{1}{m}E)$ is klt for large $m$, we may assume $-K_{X}-\Delta$ is ample.
Now the statement follows from the cone theorem.
(3) If a $\Q$-divisor is numerically trivial, then it is semi-ample by (1) and therefore $\Q$-linearly equivalent to zero.
\end{proof}

\begin{ex}
Let $X$ be a projective toric variety constructed from a fan.
Then $X$ has endomorphisms of the form $\sigma \circ f$ where $f$ is an endomorphism coming from an endomorphism of the fan
and $\sigma$ is an automorphism of $X$.
Note that the automorphism group of $X$ is a linear algebraic group and could be larger than the torus.
Conjecture \ref{shibata} and Conjecture \ref{ksc} for endomorphisms of such form, which is now true from Theorem \ref{ksc for semi-ample},
do not seem to be already known.
\end{ex}

\subsection{Construction of nef canonical heights.}\label{constr of can ht}

Let $X$ be a normal projective variety over $\QQ$.
Assume
\begin{itemize}
\item the nef cone is generated by finitely many semi-ample integral divisors;
\item  $(\Pic X)_{\Q}=N^{1}(X)_{\Q}$.
\end{itemize}
Let $f \colon X \longrightarrow X$ be a surjective morphism with $\delta_{f}>1$.
{\bf Assume every extremal ray of $\Nef(X)$ is preserved by $f^{*}$.}
(By Lemma \ref{integral eigen}, this is the case for some iterate of every $f$.)

Let $\R_{\geq0}D_{1}, \dots, \R_{\geq0}D_{r}$ be the all extremal rays of $\Nef(X)$ where
$D_{i}$ is a semi-ample integral divisor.
We may assume the following:
\begin{itemize}
\item $f^{*}D_{i} \sim \lambda_{i}D_{i}$ for some positive integer $ \lambda_{i}$;
\item $ \delta_{f}=\lambda_{1}=\cdots = \lambda_{s}> \lambda_{s+1}\geq \cdots \geq \lambda_{r}$;
\item $D_{i}$'s are base point free;
\item Let  $\pi_{i} \colon X \longrightarrow Y_{i}$ be the semi-ample fibration of $D_{i}$.
Then $\dim Y_{i}>0$ and there exists an ample divisor $H_{i}$ on $Y_{i}$ such that $\pi_{i}^{*}H_{i}\sim D_{i}$.
\end{itemize}

\begin{lem}
The morphism $f$ induces a surjective morphism $g_{i} \colon Y_{i} \longrightarrow Y_{i}$ such that $g_{i}\circ \pi_{i}=\pi_{i}\circ f$
 and  $g_{i}^{*}H_{i}\sim \lambda_{i}H_{i}$.
\end{lem}
\begin{proof}
Since $f^{*}D_{i}\sim \lambda_{i}D_{i}$, $f^{*}$ induces a linear map $H^{0}(D_{i}) \longrightarrow H^{0}( \lambda_{i}D_{i})$.
Thus there exists a rational map $g_{i} \colon Y_{i} \dashrightarrow Y_{i}$ such that $g_{i}\circ \pi_{i}=\pi_{i} \circ f$.
Consider the following diagram:
\[
\xymatrix{
\widetilde{X} \ar[d]_{p} \ar[ddr]^{q}& & \\
X \ar[rr]^{f} \ar[dd]_{\pi_{i}}& & X \ar[dd]_{\pi_{i}}\\
& \Gamma_{g_{i}} \ar[ld]_{ \alpha} \ar[rd]^{\beta}& \\
Y_{i} \ar@{-->}[rr]_{g_{i}}&& Y_{i}
}
\]
where $\Gamma_{g_{i}}$ is the graph of $g_{i}$, $p$ is a proper birational morphism.
Take any closed point $y\in Y_{i}$.
Then
\begin{align*}
\beta( \alpha^{-1}(y))&=\beta\circ q(q^{-1}( \alpha^{-1}(y)))=\pi_{i}\circ f\circ p(q^{-1}( \alpha^{-1}(y)))\\
&=\pi_{i}\circ f\circ p(p^{-1}(\pi_{i}^{-1}(y)))=\pi_{i}\circ f\circ (\pi_{i}^{-1}(y)).
\end{align*}
For every irreducible curve $C\subset \pi_{i}^{-1}(y)$, we have
\[
(D_{i}\cdot f_{*}[C])=(f^{*}D_{i}\cdot C)= \lambda_{i}(D\cdot C)=0.
\]
Since $\pi_{i}$ is the contraction of $ \overline{NE}(X)\cap D_{i}^{\perp}$,
$f(C)$ is contained in a fiber of $\pi_{i}$.
Hence $\beta( \alpha^{-1}(y))=\pi_{i}\circ f\circ (\pi_{i}^{-1}(y))$ is a point.
(Note that $\pi_{i}^{-1}(y)$ is connected.)
By Zariski main theorem, $g_{i}$ is a morphism.
Finally, 
\[
\pi_{i}^{*}g_{i}^{*} \mathcal{O}_{Y_{i}}(H_{i}) \simeq f^{*}\pi_{i}^{*} \mathcal{O}_{Y_{i}}(H_{i})
\simeq f^{*} \mathcal{O}_{X}(D_{i}) \simeq \mathcal{O}_{X}( \lambda_{i}D_{i}) \simeq \pi_{i}^{*} \mathcal{O}_{Y_{i}}( \lambda_{i}H_{i})
\]
implies $g_{i}^{*} \mathcal{O}_{Y_{i}}(H_{i}) \simeq \mathcal{O}_{Y_{i}}( \lambda_{i}H_{i})$.
\end{proof}

For $i$ with $ \lambda_{i}>1$, define
\begin{align*}
&\hat{h}_{D_{i}}(x)=\lim_{n\to \infty}\frac{h_{D_{i}}(f^{n}(x))}{ \lambda_{i}^{n}}\ \ \text{for $x\in X(\QQ)$};\\
&\hat{h}_{H_{i}}(y)=\lim_{n\to \infty}\frac{h_{H_{i}}(g_{i}^{n}(y))}{ \lambda_{i}^{n}}\ \ \text{for $y\in Y_{i}(\QQ)$}.
\end{align*}
Then they satisfy the following:
\begin{itemize}
\item $\hat{h}_{D_{i}}\circ f=\lambda_{i}\hat{h}_{D_{i}}$, $\hat{h}_{H_{i}}\circ g_{i}=\lambda_{i}\hat{h}_{H_{i}}$; 
\item $\hat{h}_{D_{i}}=\hat{h}_{H_{i}}\circ \pi_{i}\geq0$;
\item $\hat{h}_{H_{i}}=h_{H_{i}}+O(1)$ (i.e. $\hat{h}_{H_{i}}$ is a Weil height function associated with the ample divisor $H_{i}$).
\end{itemize}

\subsection{Proof of Theorem \ref{ksc for semi-ample}}
Now return to the general case.

\begin{lem}\label{integral eigen}
Let $X$ be a normal projective variety and $f \colon X \longrightarrow X$ a surjective morphism both defined over $\QQ$.
Assume that the nef cone of $X$ is a rational polyhedral cone.
Then there exists a positive integer $n$ such that $(f^{n})^{*}$ fixes all extremal rays of $\Nef(X)$.
\end{lem}
\begin{proof}
Since $f^{*} \colon N^{1}(X)_{\R} \longrightarrow N^{1}(X)_{\R}$ induces a self-bijection of $\Nef(X)$,
it induces a permutation of the extremal rays of $\Nef(X)$.
Thus, there exists a positive integer $n$ such that $(f^{n})^{*}$ fixes all extremal rays.
\end{proof}

\begin{prop}\label{van locus of ample can ht}
Let $X$ be a normal projective variety over $\QQ$.
Assume
\begin{itemize}
\item the nef cone is generated by finitely many semi-ample integral divisors;
\item  $(\Pic X)_{\Q}=N^{1}(X)_{\Q}$.
\end{itemize}
Let $f \colon X \longrightarrow X$ be a surjective morphism with $\delta_{f}>1$.
Take sufficiently large number field $k$ so that $X, f, \dots$ are all defined.
(We write the model $X_{k}$, etc.)
Then, for any number field $K\supset k$, there exists a proper closed subset $V\subset X_{K}$ such that
\[
Z_{f}(K)=\{x\in X_{k}(K)\mid \underline{h}_{f}(x)=0\}=V(K).
\]
\end{prop}

\begin{proof}
Let $N$ be a positive integer.
By easy calculation, $\underline{h}_{f}(x)=0$ if and only if 
$\underline{h}_{f^{N}}(f^{r}(x))=0$ for some $r\in \{0,1,\dots, N-1\}$.
Therefore, by replacing $f$ by its iterate, we may assume every extremal ray of $\Nef(X)$ is preserved by $f^{*}$.
Now we use the notation in \S\ref{constr of can ht}.
Let $H=D_{1}+\cdots+D_{r}$. Since $D_{i}$'s generate the nef cone, $H$ is ample.
If $ \lambda_{i}<\delta_{f}$, then $\lim_{n\to \infty }h_{D_{i}}(f^{n}(x))/\delta_{f}^{n}=0$.
If $ \lambda_{i}=\delta_{f}$, then $\lim_{n\to \infty }h_{D_{i}}(f^{n}(x))/\delta_{f}^{n}$ exists.
In particular, $l_{f}=0$ and
\begin{align*}
\liminf_{n\to \infty}\frac{h_{H}(f^{n}(x))}{\delta_{f}^{n}}=\liminf_{n\to \infty}\sum_{i=1}^{r}\frac{h_{D_{i}}(f^{n}(x))}{\delta_{f}^{n}}
=\sum_{i=1}^{s}\hat{h}_{D_{i}}(x)=\sum_{i=1}^{s}\hat{h}_{H_{i}}(\pi_{i}(x)).
\end{align*}

Therefore
\begin{align*}
Z_{f}(K)=\bigcap_{i=1}^{s}\pi_{i}^{-1}(\{y\in (Y_{i})_{k}(K)\mid \hat{h}_{H_{i}}(y)=0\}).
\end{align*}
Since $\hat{h}_{H_{i}}$ is an ample height, the set $\{y\in (Y_{i})_{k}(K)\mid \hat{h}_{H_{i}}(y)=0\}$ is finite.
Since $\dim Y_{i}>0$, the statement follows.
\end{proof}

\begin{proof}[Proof of Theorem \ref{ksc for semi-ample}]
By Proposition \ref{van locus of ample can ht}, Conjecture \ref{shibata} holds for $f$.
The last assertion follows from Remark \ref{shibata implies ksc}.
\end{proof}

\section{Equivariant MMP and KSC}\label{sec:equiMMP}

{\bf Convention}: In this section, KSC for $f$, where $f$ is an endomorphism of a normal projective variety over $\QQ$,
means Conjecture \ref{ksc} for $f$.
By Remark \ref{rmk:exist-of-ad}, this is equivalent to say that $ \alpha_{f}(x) = \delta_{f}$
for every point $x\in X( \overline{\mathbb Q})$ with  Zariski dense $f$-orbit.

\subsection{Equivariant MMP}

\begin{defn}
A surjective endomorphism $f \colon X \longrightarrow X$ of normal projective variety $X$ is called int-amplified if 
there exists an ample Cartier divisor $H$ on $X$ such that $f^{*}H-H$ is ample.
\end{defn}

We collect basic properties of int-amplified endomorphisms in the following lemma.

\begin{lem}\label{lem:intamp}\ 
\begin{enumerate}
\item Let $X$ be a normal projective variety, $f \colon X \longrightarrow X$ a surjective morphism, and $n>0$ a positive integer.
Then, $f$ is int-amplified if and only if so is $f^{n}$.

\item Let $\pi \colon X \longrightarrow Y$ be a surjective morphism between normal projective varieties.
Let $f \colon X \longrightarrow X$, $g \colon Y \longrightarrow Y$ be surjective endomorphisms such that $\pi \circ f=g \circ \pi$.
If $f$ is int-amplified, then so is $g$. 

\item  Let $\pi \colon X \dashrightarrow Y$ be a dominant rational map between normal projective varieties of same dimension.
Let $f \colon X \longrightarrow X$, $g \colon Y \longrightarrow Y$ be surjective endomorphisms such that $\pi \circ f=g \circ \pi$.
Then $f$ is int-amplified if and only if so is $g$.

\item If a normal $\Q$-factorial projective variety $X$ admits an int-amplified endomorphism, then $-K_{X}$ is numerically equivalent to an effective $\Q$-divisor.
In particular, if $X$ is rationally connected, then $-K_{X}$ is $\Q$-linearly equivalent to an effective $\Q$-divisor.

\end{enumerate}
\end{lem}
\begin{proof}
See \cite[Lemma 3.3, 3,5, 3.6, Theorem 1.5]{meng}.
\end{proof}

Meng and Zhang established minimal model program equivariant with respect to endomorphisms, for varieties admitting an int-amplified endomorphism.
We summarize their results that we need.

\begin{thm}[Meng-Zhang]\label{equiv-thm}
Let $X$ be a $\Q$-factorial Kawamata log terminal (klt) projective variety over $\QQ$ admitting an int-amplified endomorphism.
\begin{enumerate}
\item\label{fin-st} There are only finitely many $K_{X}$-negative extremal rays of $ \overline{NE}(X)$.
Moreover, let $f \colon X \longrightarrow X$ be a surjective endomorphism of $X$.
Then every $K_{X}$-negative extremal ray is fixed by the linear map $(f^{n})_{*}$ for some $n>0$.
\item\label{end-induced}  Let $f \colon X \longrightarrow X$ be a surjective endomorphism of $X$.
Let $R$ be a $K_{X}$-negative extremal ray and $\pi \colon X \longrightarrow Y$ its contraction.
Suppose $f_{*}(R)=R$.
Then,
\begin{enumerate}
\item \label{induced-on-contr} $f$ induces an endomorphism $g \colon Y \longrightarrow Y$ such that $g\circ \pi=\pi \circ f$;
\item \label{induced-on-flip}if $\pi$ is a flipping contraction and $X^{+}$ is the flip, the induced rational self-map $h \colon X^{+} \dashrightarrow X^{+}$
is a morphism.
\end{enumerate}

\end{enumerate}
\end{thm}
\begin{proof}
(\ref{fin-st}) is a special case of  \cite[Theorem 4.6]{meng-zhang2}.
(\ref{induced-on-contr}) is true since the contraction is determined by the ray $R$.
(\ref{induced-on-flip}) follows from \cite[Lemma 6.6]{meng-zhang}.
\end{proof}

\begin{thm}[Equivariant MMP (Meng-Zhang)]\label{equiMMP}
Let $X$ be a $\Q$-factorial klt projective variety over $\QQ$ admitting an int-amplified endomorphism.
Then for any surjective endomorphism $f \colon X \longrightarrow X$, there exists a positive integer $n>0$ and 
a sequence of rational maps
\begin{align*}
X=X_{0} \dashrightarrow X_{1} \dashrightarrow \cdots \dashrightarrow X_{r}
\end{align*}
such that
\begin{enumerate}
\item $X_{i} \dashrightarrow X_{i+1}$ is either the divisorial contraction, flip, or Fano contraction of a 
$K_{X_{i}}$-negative extremal ray;
\item $X_{r}$ is a $Q$-abelian variety (i.e. there exists a quasi-\'etale finite surjective morphism $A \longrightarrow X_{r}$ from an abelian variety $A$.
Note that $X_{r}$ might be a point);
\item there exist surjective endomorphisms $g_{i} \colon X_{i} \longrightarrow X_{i}$ for $i=0, \dots ,r$
such that $g_{0}=f^{n}$ and the following diagram commutes:
\[
\xymatrix{
X_{i} \ar@{-->}[r] \ar[d]_{g_{i}} & X_{i+1} \ar[d]^{g_{i+1}}\\
X_{i} \ar@{-->}[r] & X_{i+1} ;
}
\]
\item \label{endonqabel} there exists a quasi-\'etale finite surjective morphism $A \longrightarrow X_{r}$ from an abelian variety $A$
and an surjective endomorphism $h \colon A \longrightarrow A$ such that the diagram
\[
\xymatrix{
A \ar[r]^{h} \ar[d] & A \ar[d]\\
X_{r} \ar[r]_{g_{r}} & X_{r} ;
}
\]
commutes.
\end{enumerate}
\end{thm}
\begin{proof}
This is a part of \cite[Theorem 1.2]{meng-zhang2}.
\end{proof}

\begin{rmk}
Surjective endomorphisms on a Q-abelian variety always lift to a certain quasi-\'etale cover by an abelian variety.
See \cite[Lemma 8.1 and Corollary 8.2]{cmz}, for example. The proof works over any algebraically closed field.
\end{rmk}

\subsection{Easy corollaries of equivariant MMP}

\begin{lem}\label{bir and ksc}
Consider the commutative diagram
\[
\xymatrix{
X \ar[r]^{f} \ar@{-->}[d]_{\pi}& X \ar@{-->}[d]^{\pi}\\
Y \ar[r]_{g}& Y
}
\]
where $X, Y$ are normal projective varieties over $\QQ$, $f,g$ are surjective morphisms,
and $\pi$ is a birational map.
Then KSC holds for $f$ if and only if so does for $g$.
\end{lem}
\begin{proof}
Assume KSC for $g$.
Let $x\in X(\QQ)$ be a point whose $f$-orbit is Zariski dense.
We need to show that $ \alpha_{f}(x)=\delta_{f}$.
There exists a sequence $n_{1}<n_{2} <\cdots $ such that $\{f^{n_{i}}(x) \mid i=1,2,\dots\}$ is Zariski dense and
$f^{n_{i}}(x) \notin I_{\pi}$ where $I_{\pi}$ is the indeterminacy locus of $\pi$.
Replacing $x$ by $f^{n_{1}}(x)$, we may assume $x\notin I_{\pi}$.
Then we get $\pi(f^{n_{i}}(x))=g^{n_{i}}(\pi(x))$ and $\{g^{n_{i}}(\pi(x))\mid i=1,2, \dots\}$ is Zariski dense in $Y$.
By KSC for $g$, we get $ \alpha_{g}(\pi(x))=\delta_{g}=\delta_{f}$.

Take a birational morphism $q \colon Z \longrightarrow X$ from a normal projective variety $Z$ such that 
$q$ is isomorphic over $X\setminus I_{\pi}$ and $\pi\circ q=:p$ becomes a morphism.
Fix an ample divisor $H$ on $Y$.
Since $q^{*}q_{*}p^{*}H-p^{*}H$ is a $q$-exceptional effective divisor, we get
$h_{H}\circ \pi \leq h_{q_{*}p^{*}H}+O(1)$ on $(X\setminus I_{\pi})(\QQ)$.
Thus, $h_{H}(g^{n_{i}}(\pi(x)))=h_{H}(\pi(f^{n_{i}}(x)))\leq h_{q_{*}p^{*}H}(f^{n_{i}}(x))+O(1)$.
Since we know that the arithmetic degree exists, we get $ \alpha_{f}(x)\geq \alpha_{g}(\pi(x))=\delta_{f}$.
\end{proof}

\begin{lem}\label{aut of mfs}
Consider the commutative diagram of surjective morphisms 
\[
\xymatrix{
X \ar[r]^{f} \ar[d]_{\pi}& X \ar[d]^{\pi}\\
Y \ar[r]_{g}& Y
}
\]
where $X, Y$ are normal projective varieties, $f$ is an automorphism,
and $\pi_{*} \colon N_{1}(X)_{\Q} \longrightarrow N_{1}(Y)_{\Q}$ is surjective with $\dim \Ker \pi_{*}=1$ (e.g. a Fano contraction). 
Then $\delta_{f}=\delta_{g}$.
In particular, KSC for $g$ implies KSC for $f$.
\end{lem}
\begin{proof}
Since $f_{*} \colon N_{1}(X)_{\Z} \longrightarrow N_{1}(X)_{\Z} $ is an automorphism, 
the eigenvalues of $f_{*}$ are unit algebraic integers.
Thus, $f_{*}|_{\Ker \pi_{*}}$ is $\pm \id$. 
Therefore, $\delta_{f}=\text{spectral radius of $f_{*}$}=\text{spectral radius of $g_{*}$}=\delta_{g}$.
\end{proof}

\begin{thm}
Let $X$ be a $\Q$-factorial klt projective variety over $\QQ$ admitting an int-amplified endomorphism.
Let $f \colon X \longrightarrow X$ be a surjective endomorphism.
Assume one of the following:
\begin{enumerate}
\item $f$ is an automorphism,
\item $f^{n}$ are primitive for all $n\geq1$ (i.e. there is no dominant rational map $p \colon X \dashrightarrow Y$ and dominant rational self-map $g \colon Y \dashrightarrow Y$
with $0<\dim Y< \dim X$ such that $p\circ f^{n}=g \circ p$), or
\item the Kodaira dimension of $X$ is non-negative: $\kappa(X)\geq 0$.
\end{enumerate}
Then KSC holds for $f$.
\end{thm}
\begin{proof}
(1) follows from Theorem \ref{equiMMP}, Lemma \ref{bir and ksc} \ref{aut of mfs} and KSC for abelian varieties (\cite{ks1}, \cite{sil2}).
(2) By the assumption, the only Fano contraction that can appear during the equivariant MMP in Theorem \ref{equiMMP} is contraction to a point.
Thus KSC for $f$ follows from Lemma \ref{bir and ksc} and KSC for normal varieties of Picard number one and KSC for abelian varieties.
(3) By the assumption, the output of equivariant MMP in Theorem \ref{equiMMP} is a Q-abelian variety and thus the statement follows from KSC for abelian varieties.
(Actually, $X$ itself is a Q-abelian variety in this case.)
\end{proof}

\subsection{A reduction of KSC using equivariant MMP}

In this section, we propose a possible strategy to prove KSC for endomorphisms on varieties admitting an int-amplified endomorphism by dimension induction.
Let $d$ be a non-negative integer and consider the following conditions.
\vspace{8pt}
\begin{notation}
\item[$(*)_{d}$]
Let $X$ be an arbitrary $\Q$-factorial klt projective variety of dimension $d$ admitting an int-amplified endomorphism.
Let $f \colon X \longrightarrow X$ be an arbitrary surjective morphism.
Then, $ \alpha_{f}(x)=\delta_{f}$ for every $x\in X(\QQ)$ with Zariski dense $f$-orbit.
\end{notation}
\vspace{8pt}
\begin{notation}
\item[${\rm(Mfs)}_{d}$]
Let $X$ be an arbitrary $\Q$-factorial klt projective variety of dimension $d$ admitting an int-amplified endomorphism.
Suppose there exists a unique $K_{X}$-negative extremal ray and its contraction $\pi \colon X \longrightarrow Y$
is a Fano contraction.
Let $f \colon X \longrightarrow X$ be an arbitrary surjective morphism which induces a surjective endomorphism $g \colon Y \longrightarrow Y$
with $\delta_{f}>\delta_{g}$.
Then, $ \alpha_{f}(x)=\delta_{f}$ for every $x\in X(\QQ)$ with Zariski dense $f$-orbit.
\end{notation}
\vspace{8pt}

We also need the following purely birational geometric condition, which is largely expected to hold.
\vspace{8pt}
\begin{notation}
\item[$ {\rm (TF)}_{d}$]
There is no infinite sequence of $K_{X}$-flips for any $\Q$-factorial klt projective variety $X$ of dimension $d$.
\end{notation}
\vspace{8pt}
It is known that ${\rm (TF)}_{d}$ is true for $d\leq3$ (cf. \cite[Theorem 5.2]{shok}). 

\begin{prop}\label{reduction-by-mmp}
Let $d$ be a positive integer and assume $ {\rm (TF)}_{d}$.
Let $X$ be a $\Q$-factorial klt projective variety of dimension $d$ admitting an int-amplified endomorphism.
Let $f \colon X \longrightarrow X$ be a surjective morphism.
We can construct a sequence of flips or divisorial contractions of $K$-negative extremal rays,
which is equivariant with respect to $f^{n}$ for some $n>0$, with output $f' \colon X' \longrightarrow X'$ satisfying one of the following:
\begin{enumerate}
\item \label{nefcase} $X'$ is a Q-abelian variety;
\item \label{redtobase} There is a Fano contraction $\pi \colon X' \longrightarrow Y$ and $f'$ induces an endomorphism $g \colon Y \longrightarrow Y$ with
$\delta_{f'}= \delta_{g}$;
\item There exists a unique $K_{X'}$-negative extremal ray and its contraction $\pi \colon X' \longrightarrow Y$ is a Fano contraction.
Moreover, $f'$ induces an endomorphism $g \colon Y \longrightarrow Y$ and $\delta_{f'}>\delta_{g}$.
\end{enumerate}

In (\ref{nefcase}), KSC holds for $f$.
In (\ref{redtobase}), KSC for $f$ follows from KSC for $g$.

In particular, $(*)_{k}$ for $k=0,1, \dots, d-1$ and $ {\rm (Mfs)}_{d}$ implies $(*)_{d}$.
\end{prop}
\begin{proof}

By Theorem \ref{equiMMP}, we only need to consider the case where the output of the equivariant MMP is a Fano contraction 
 $\pi \colon X' \longrightarrow Y$ with $\delta_{f'}>\delta_{g}$.
 In this case, the generalized eigenspace of $(f')_{*} \colon N_{1}(X')_{\R} \longrightarrow N_{1}(X')_{\R}$ with eigenvalue $\delta_{f'}$
 is the kernel of $\pi_{*}$ and one dimensional.
 Therefore, if there is another $K_{X'}$-negative Fano contraction, we get case (\ref{redtobase}).
 Assume in every step of equivariant MMP, 
 there are at least two $K$-negative extremal rays and there is at most one $K$-negative Fano contraction, or
 the $K$-negative extremal ray is unique and its contraction is birational.
 Then, we get an infinite sequence of flips and this contradicts to our assumption.
 Thus, we get one of the case in the Proposition.
 
 In the case (\ref{nefcase}), KSC for $f$ is true by Theorem \ref{equiMMP} (\ref{endonqabel}), Lemma \ref{bir and ksc}, and
 KSC for abelian varieties.
 In the case (\ref{redtobase}), KSC for $f$ follows from KSC for $f'$ by Lemma \ref{bir and ksc}, and KSC for $f'$ follows from
 that of $g$ since $\delta_{f}=\delta_{g}$.
\end{proof}

\begin{prop}\label{nonvan}
Let $d$ be a positive integer and assume $ {\rm (TF)}_{i}$ for $i\leq d$.
Let $X$ be a $\Q$-factorial klt rationally connected projective variety admitting an int-amplified endomorphism.
Then, for every nef $\Q$-Cartier divisor $D$ on $X$, we have $\kappa(D)\geq0$, i.e. 
$H^{0}(X, \mathcal{O}_{X}(mD))\neq 0$ for some positive integer $m$.
\end{prop}
\begin{proof}
We prove by induction on the dimension.
By Theorem \ref{equiv-thm}, there are only finitely many $K_{X}$-negative extremal rays $R_{1}, \dots , R_{t}$ of $ \overline{NE}(X)$.
Take a non-negative rational number $r$ such that
\begin{align*}
&((D+rK_{X})\cdot R_{i})\geq0 \qquad &\text{for all $i$};\\
&((D+rK_{X})\cdot R_{i_{0}})=0 \qquad &\text{for some $i_{0}$}.
\end{align*}
Note that $D+rK_{X}$ is nef.
Let $\pi \colon X \longrightarrow Y$ be the contraction of $R_{i_{0}}$.

(1) When $\pi$ is a Fano contraction, take a $\Q$-Cartier divisor $L$ on $Y$ such that $\pi^{*}L \sim_{\Q} D+rK_{X}$.
Note that  $L$ is nef and $Y$ is a $\Q$-factorial klt rationally connected variety admitting an int-amplified endomorphism (cf. Lemma \ref{lem:intamp} and Theorem \ref{equiv-thm}). 
Thus, by induction hypothesis, $\kappa(L)\geq0$.
Since $\kappa(-K_{X})\geq0$ (Lemma \ref{lem:intamp}), we get $\kappa(D)=\kappa(\pi^{*}L-rK_{X})\geq0$.

(2) When $\pi$ is a divisorial contraction, take a $\Q$-Cartier divisor $D_{Y}$ on $Y$ such that $\pi^{*}D_{Y} \sim_{\Q} D+rK_{X}$.
Then $D_{Y}$ is also nef and $Y$ is a $\Q$-factorial klt rationally connected variety admitting an int-amplified endomorphism.
Since $\kappa(-K_{X})\geq0$, we have $\kappa(D)\geq \kappa(D_{Y})$ (i.e. Proposition \ref{nonvan} for $(X, D)$ follows from that of $(Y, D_{Y})$).
Replace $(X, D)$ with $(Y, D_{Y})$ and do the above procedure.

(3) When $\pi$ is a flipping contraction, take a $\Q$-Cartier divisor $D_{Y}$ on $Y$ such that $\pi^{*}D_{Y} \sim_{\Q} D+rK_{X}$.
Let $\pi^{+} \colon X^{+} \longrightarrow Y$ be the flip.
Then $D^{+}:=(\pi^{+})^{*}D_{Y}$ is also nef and $X^{+}$ is a $\Q$-factorial klt rationally connected variety admitting an int-amplified endomorphism.
Since $\kappa(-K_{X})\geq0$, Proposition \ref{nonvan} for $(X, D)$ follows from that of $(X^{+}, D^{+})$.
Replace $(X, D)$ with $(X^{+}, D^{+})$ and do the above procedure.

Since divisorial contraction occurs only finitely many times and we assume there is no infinite sequence of flips, we finally get a Fano contraction (i.e. case (1)) or
minimal model (i.e. $K_{X}$ becomes nef).
Therefore, to end the proof, we may assume $K_{X}$ is nef.
Since $\kappa(-K_{X})\geq0$, we have $K_{X}\sim_{\Q}0$.
By \cite[Theorem 5.2]{meng}, $X$ is a Q-abelian variety.
Let $q \colon A \longrightarrow X$ be a quasi-\'etale finite morphism from an abelian variety $A$.
Since every nef integral divisor on an abelian variety is numerically equivalent to an effective divisor 
(cf. \cite[Lemma 1.1]{bau} for a proof over $\C$. 
The case over an arbitrary algebraically closed field of characteristic zero follows from the Lefschetz principle and by considering symmetric part of the nef divisor.), 
there exists a $\Q$-Cartier effective divisor $E$ on $A$ such that $q^{*}D\equiv E$.
Then $q_{*}(q^{*}D-E)$ is numerically equivalent to zero as a cycle in the sense of Fulton (\cite[Chapter 19]{ful}). 
By \cite[Lemma 2.12]{meng-zhang},
$0\equiv q_{*}(q^{*}D-E)=(\deg q) D-q_{*}E$ as $\Q$-divisors.
Since $X$ is rationally connected, we get $(\deg q)D \sim_{\Q} q_{*}E$ and this is what we wanted.

\end{proof}

\begin{thm}\label{KSCfork>0}
Let $d$ be a positive integer and assume $ {\rm (TF)}_{i}$ for $i\leq d$ and $(*)_{j}$ for $j\leq d-1$.
Let $X$ be a $\Q$-factorial klt rationally connected projective variety of dimension $d$ admitting an int-amplified endomorphism.
Suppose $\kappa(-K_{X})>0$.
Then KSC holds for every surjective endomorphism $f \colon X \longrightarrow X$.
\end{thm}
\begin{proof}
By Lemma \ref{bir and ksc} and Proposition \ref{reduction-by-mmp}, we may assume
\begin{itemize}
\item there exists a unique $K_{X}$-negative extremal ray $R$ and its contraction $\pi \colon X \longrightarrow Y$ is a Fano contraction;
\item $f$ induces an endomorphism $g \colon Y \longrightarrow Y$;
\item $\delta_{f}>\delta_{g}$.
\end{itemize}

Since $\delta_{f}>\delta_{g}$ and $\pi$ is a ray contraction, there exists a nef integral divisor $D\not\equiv 0$ such that
$f^{*}D \sim_{\Q} \delta_{f}D$ and $(D\cdot R)>0$.
Since $(K_{X}\cdot R)<0$, there exists a positive rational number $r$ such that $((D+rK_{X})\cdot R)=0$.
Since $R$ is the unique $K_{X}$-negative extremal ray, $D+rK_{X}$ is nef.
By cone theorem, there exists a nef $\Q$-Cartier divisor $L$ on $Y$ such that $\pi^{*}L\sim_{\Q} D+rK_{X}$.

By Lemma \ref{lem:intamp} and Theorem \ref{equiv-thm}, $Y$ also admits an int-amplified endomorphism, and it is $\Q$-factorial klt rationally connected.
By Proposition \ref{nonvan}, $\kappa(L)\geq0$. 

Since flips and divisorial contractions preserve the condition "$\kappa(-K_{X})>0$", 
we get $\kappa(D)=\kappa(\pi^{*}L-rK_{X})\geq \kappa(-K_{X})>0$.
Now, the theorem follows from Proposition \ref{canht-posiitaka}.

\end{proof}

\begin{prop}\label{ksc-for-klt-surf}
$(*)_{d}$  is true for $d\leq 2$.
\end{prop}
\begin{proof}
Let $X$ be a $\Q$-factorial klt surface admitting an int-amplified endomorphism.
By \cite{mayo}, 
the output $X_{r}$ of equivariant MMP as in Theorem \ref{equiMMP}, is either
\begin{enumerate}
\item a Q-abelian surface and the induced endomorphism lifts to a quasi-\'etale cover by an abelian surface;
\item a smooth projective surface;
\item a rational surface with a quasi-\'etale cover by an smooth projective surface and the induced endomorphism on $X_{r}$ lifts to the cover;
\item\label{ratsurfk>0} a rational surface with $\kappa(-K_{X_{r}})>0$;
\item\label{surfp=1} a surface of Picard number one.
\end{enumerate}
In the first three cases, KSC for endomorphisms on $X_{r}$ follows from KSC for endomorphisms on smooth projective surfaces,
which is proved in \cite{mss}.
In the case (\ref{ratsurfk>0}), KSC follows from Theorem \ref{KSCfork>0}.
In the case (\ref{surfp=1}), endomorphism on $X_{r}$ is polarized and therefore KSC is true.

Finally, KSC for endomorphisms on $X$ follows from that of on $X_{r}$ by  Lemma \ref{bir and ksc}.
\end{proof}

\begin{cor}
Let $X$ be a rationally connected $\Q$-factorial klt projective variety of dimension $3$ admitting an int-amplified endomorphism.
Suppose $\kappa(-K_{X})>0$.
Then, KSC holds for every surjective morphism $f \colon X \longrightarrow X$.
\end{cor}

\subsection{An example}
In this subsection, we give a typical example of int-amplified endomorphism.
Although KSC for this example follows from that of semi-abelian varieties, which is proved in \cite{ms},
this example demonstrates how equivariant MMP would work to prove KSC.

Let $G$ be a semi-abelian variety over $ \QQ$ or an algebraically closed field of characteristic zero.
Let $G \longrightarrow A$ be the projection and suppose $\dim A>0$.
By \cite{vo}, there is a compactification of the form 
$G \subset Z=\P( \mathcal{O}_{A} \oplus \mathcal{M}_{1})\times_{A} \cdots \times_{A} \P(\mathcal{O}_{A} \oplus \mathcal{M}_{r})$ 
where $ \mathcal{M}_{i}$ are algebraically zero line bundles on $A$.
Moreover, for any integer $n \neq 0$, the multiplication map by $n$ on $G$ extends to a morphism $[n]$ on $Z$.
Let $X, Y$ be the quotients of $Z, A$ by $[-1]$.
Let us consider the following commutative diagram:

\[
\xymatrix{
[n]\acts Z \ar@<3.0ex>[d]_{p} \ar[r]^(.4){h} & Z/\left<-1\right> =X \ar@<-5.2ex>[d]^{\pi} \racts f\\
[n]\acts A \ar[r]& A/ \left<-1\right>=Y \racts g
}
\]
where $p$ is the projection, $h$ is the quotient map, and $f$ and $g$ are the endomorphisms induced by the multiplication map $[n]$.

\begin{prop}
Notation as above.
The following hold:
\begin{enumerate}
\item $X$ is a klt $\Q$-factorial normal projective variety, $f$ is an int-amplified endomorphism;
\item $h$ is a quasi-\'etale finite morphism;
\item 
	\begin{enumerate}
	\item If $ \mathcal{M}_{1}, \dots , \mathcal{M}_{r}$ are linearly independent over $\Z$ as elements in $\Pic^{0}(A)$, 
		  then $\kappa(-K_{X})=0$.
	\item If $ \mathcal{M}_{1}, \dots , \mathcal{M}_{r}$ are linearly dependent over $\Z$ as elements in $\Pic^{0}(A)$, 
		  then $\kappa(-K_{X})>0$.
	\end{enumerate}
\end{enumerate}
\end{prop}
\begin{proof}
(1) Since $X$ is a quotient of a smooth projective variety by a finite group, the first statement is clear.
The induced morphism $f$ is int-amplified because so is $[n] \colon Z \longrightarrow Z$.

(2) The fixed points locus $Z^{[-1]}$ is contained in the fiber of $p$ over $2$-torsion points of $A$ and
nowhere dense in these fibers.
Thus, it has codimension at least two, which implies $h$ is quasi-\'etale.

(3) Since $h$ is quasi-\'etale, we have $h^{*}(-K_{X})\sim -K_{Z}$ and therefore $\kappa(-K_{X})=\kappa(-K_{Z})$.
Let $\pr_{i} \colon Z \longrightarrow \P( \mathcal{O}_{A} \oplus \mathcal{M}_{i})$ and
$p_{i} \colon \P( \mathcal{O}_{A} \oplus \mathcal{M}_{i}) \longrightarrow A$ be the projections.
By the canonical bundle formula of projective bundles , we have
\[
\mathcal{O}_{Z}(K_{Z})= \bigotimes_{i=1}^{r}\pr_{i}^{*}( \mathcal{O}(-2) {\otimes} p_{i}^{*} \mathcal{M}_{i}).
\]
By the K\"unneth formula and the  projection formula, for any $m>0$ we have
\begin{align*}
&H^{0}(Z, \mathcal{O}_{Z}(-mK_{Z})) = H^{0}(A, p_{*}\mathcal{O}_{Z}(-mK_{Z}))\\
&\simeq H^{0}(A, \bigotimes_{i=1}^{r} {p_{i}}_{*}( \mathcal{O}(2m) \otimes p_{i}^{*} \mathcal{M}_{i}^{-m}))
\simeq H^{0}(A, \bigotimes_{i=1}^{r} ( \Sym^{2m}(\mathcal{O}_{A} \oplus \mathcal{M}_{i}) \otimes  \mathcal{M}_{i}^{-m}))\\
&\simeq \bigoplus_{-m\leq i_{j} \leq m}H^{0}(A, \mathcal{M}_{1}^{i_{1}} \otimes \cdots \otimes \mathcal{M}_{r}^{i_{r}}).
\end{align*}
Since $H^{0}(A, \mathcal{M}_{1}^{i_{1}} \otimes \cdots \otimes \mathcal{M}_{r}^{i_{r}}) \neq 0$ is equivalent to 
$\mathcal{M}_{1}^{i_{1}} \otimes \cdots \otimes \mathcal{M}_{r}^{i_{r}} \simeq \mathcal{O}_{A}$, we are done.

\end{proof}

Now, the projection $\pi \colon X \longrightarrow Y$ is the composite of $r$ $K$-negative extremal ray contractions:
each contraction corresponds to the projection 
$\P( \mathcal{O}_{A} \oplus \mathcal{M}_{1})\times_{A} \cdots \times_{A} \P(\mathcal{O}_{A} \oplus \mathcal{M}_{i})
\longrightarrow \P( \mathcal{O}_{A} \oplus \mathcal{M}_{1})\times_{A} \cdots \times_{A} \P(\mathcal{O}_{A} \oplus \mathcal{M}_{i-1})$.
At every contraction, KSC for the upstairs follows from that of the downstairs because dynamical degree does not change.
Finally, we reach $g \colon Y \longrightarrow Y$, where $Y$ is a Q-abelian variety or $\P^{1}$.
Therefore, for this particular example, equivariant MMP reduces KSC for $f$ to KSC for endomorphisms on Q-abelian variety or $\P^{1}$, which can be further reduced to
abelian varieties for the former case and the latter case is trivial.

\section{Linear algebraic groups}\label{sec:grp}

In this section we prove the following theorem.

\begin{thm}\label{main: lin alg grp}
Let $G$ be a connected linear algebraic group and $f \colon G \longrightarrow G$ a surjective group endomorphism,
both defined over $ \overline{\mathbb Q}$.
If the $f$-orbit of a point $x\in G(\QQ)$ is Zariski dense , then $ \alpha_{f}(x)$ exists and is equal to $\delta_{f}$. 
\end{thm}

\subsection{Endomorphisms of linear algebraic groups}
In this subsection, the ground field $k$ is an algebraically closed field of characteristic zero. 
We refer to \cite{bor} for basic facts on linear algebraic groups.

\begin{thm}\label{dynamical deg of endom of alg group}
Let $G$ be a connected linear algebraic group and $f \colon G \longrightarrow G$
a surjective group endomorphism.
Let $B=T \ltimes U \subset G$ be a Borel subgroup of $G$ where $T$ and $U$ are maximal torus and maximal unipotent subgroup respectively.
Take $g\in G$ so that 
\[
gf(B)g^{-1}=B,\ gf(T)g^{-1}=T.
\]
Then
\[
\delta_{f}=\delta_{\Int(g)\circ f}=\delta_{(\Int(g) \circ f)|_{T}}.
\]
Here $\Int(g) \colon G \longrightarrow G; x \mapsto gxg^{-1}$ is the inner automorphism defined by $g$.
\end{thm}

\begin{proof}
{\bf Step1}\quad
First, we prove that $\delta_{f}=\delta_{\Int(g)\circ f}$ for any $g\in G$.
Take a closed embedding $i \colon G \longrightarrow \GL_{N}$ as algebraic groups.
The group $\GL_{N}$ is canonically embedded as an open subscheme of $\P^{N^{2}}$.
Let $ \overline{G}$ be the closure of $G$ in  $\P^{N^{2}}$.
Then for any $g\in G$, $\Int(g) \colon G \longrightarrow G$ is extended to $\Int(i(g)) \colon \GL_{N} \longrightarrow \GL_{N}$ and
this is extended to an automorphism $ \sigma_{g} \colon \P^{N^{2}} \longrightarrow \P^{N^{2}}$. 
In this way, the action of $G$ on itself by inner automorphisms is extended to an action on $ \overline{G}$.
Let $\pi \colon \widetilde{ G} \longrightarrow  \overline{G}$ be a $G$-equivariant resolution of singularities (cf. \cite[Proposition 3.9.1]{kol}). 
The automorphism of $ \widetilde{ G}$ induced by $\Int(g)$ is denoted by $ \widetilde{\Int(g)}$.
Set $H=\pi^{*} \mathcal{O}_{\P^{N^{2}}}|_{ \overline{G}}$. Then $( \widetilde{\Int(g)})^{*}H\sim H$.
Note that we can calculate the dynamical degrees of self maps of $ \widetilde{ G}$ by using $H$ since it is nef and big.

Let $\varphi \colon  \widetilde{G} \dashrightarrow \widetilde{G}$ be the rational map induced by $\Int(g)\circ f \colon G \longrightarrow G$, and
$F \colon \widetilde{G} \dashrightarrow \widetilde{G}$ the one induced by $f \colon G \longrightarrow G$.
Then

\begin{align*}
&(\varphi^{n})^{*}H= (\widetilde{\Int(gf(f)\cdots f^{n-1}(g))}\circ F^{n})^{*}H\\
&\sim (F^{n})^{*}(\widetilde{\Int(gf(f)\cdots f^{n-1}(g))})^{*}H\sim (F^{n})^{*}H.
\end{align*}
Thus
\begin{align*}
\delta_{f}&=\lim_{n\to \infty}(( F^{n})^{*}H\cdot H^{\dim G-1})^{1/n}\\
&=\lim_{n\to \infty}(( \varphi^{n})^{*}H\cdot H^{\dim G-1})^{1/n}=\delta_{\Int(g)\circ f}.
\end{align*}

{\bf Step2}\quad
If $G$ is unipotent, then $f$ is an automorphism and $\delta_{f}=1$.
Indeed, the kernel of $f$ is a finite subgroup of $G$ and a finite subgroup of an unipotent group is trivial.
Hence, $f$ is an automorphism.
By \cite{hocmos}, the group of automorphisms $ \Aut G$ has a structure of a linear algebraic group. 
Therefore, $\delta_{f}=1$
($f^{m} \colon G \longrightarrow G$ is equivariantly embedded into a projective space for some $m>0$. cf. \cite[Theorem 7.3]{dol}).

{\bf Step3}\quad
Let $G$ be any connected linear algebraic group and $B=T \ltimes U$ a Borel subgroup.
Assume a surjective group homomorphism $f \colon G \longrightarrow G$ preserves $B$ and $T$
, i.e. $f(B)\subset B$ and $f(T)\subset T$.

Let $R_{u}G$ be the unipotent radical of $G$ and $L\subset G$ be a Levi subgroup of $G$
(i.e. $G=L\ltimes R_{u}G$).
Then $f(L)\subset G$ is also a Levi subgroup of $G$.
Indeed, $f(L)$ and $R_{u}G$ generate $G$. To see that $f(L)$ is Levi, it is enough to show that
$f(L)\cap R_{u}G=\{1\}$.
By dimension counting, $f(L)\cap R_{u}G$ is finite.
Since any finite subgroup of an unipotent group is trivial, we get $f(L)\cap R_{u}G=\{1\}$.
Since Levi subgroups are all conjugate, there exists $g_{1}\in G$ such that
$g_{1}f(L)g_{1}^{-1}=L$.
Set $f_{1}=\Int g_{1}\circ f$.
Then $f_{1}=f_{1}|_{L}\times f_{1}|_{R_{u}G} \colon G=L\ltimes R_{u}G \longrightarrow L\ltimes R_{u}G$.
By Step1 and Step2, we get
$\delta_{f}=\delta_{f_{1}}=\max\{\delta_{f_{1}|_{L}}, \delta_{f_{1}|_{R_{u}G}}\}=\delta_{f_{1}|_{L}}$.

Take a Borel subgroup $B_{L}=T_{L} \ltimes U_{L} \subset L$ and $g_{2}\in L$ such that
$h:=\Int g_{2} \circ f_{1}|_{L}$ preserves $B_{L}$ and $T_{L}$.
Then $h$ also preserves $U_{L}$.
Let $U_{L}^{-}$ be the opposite maximal unipotent of $U_{L}$.
Then $h(U_{L}^{-})\subset U_{L}^{-}$ (look at the Lie algebras). 
Since $L$ is reductive, $U_{L}^{-}\times T_{L}\times U_{L}\subset L$ is open and
the restriction of $h$ to this open set is $h|_{U_{L}^{-}}\times h|_{T_{L}}\times h|_{U_{L}}$.
Therefore, by Step 1 and Step2, we get 
$\delta_{f}=\delta_{f_{1}|L}=\delta_{h}=\max\{\delta_{h|_{U_{L}^{-}}}, \delta_{h|_{T_{L}}}, \delta_{h|_{U_{L}}}\}=\delta_{h|_{T_{L}}}$.

Now, $B_{L}\ltimes R_{u}G$ is a Borel subgroup of $G$, $T_{L}$ is a maximal torus of $G$,
$\Int g_{2}g_{1} \circ f$ preserves them, and
$h|_{T_{L}}=(\Int g_{2}g_{1} \circ f)|_{T_{L}}$.
Thus, we have proved the following:
for any surjective group homomorphism $f \colon G \longrightarrow G$, 
there exists $g\in G$ and a Borel subgroup $B=T \ltimes U \subset G$ such that
$\Int g \circ f$ preserves $B$ and $T$, and $\delta_{f}=\delta_{(\Int g\circ f)|_{T}}$.

{\bf Step4}\quad
Let $f \colon G \longrightarrow G$ be any surjective group homomorphims.
Let $B=T\ltimes U, B'=T'\ltimes U'$ be two Borel subgroups of $G$ and
$g, g'\in G$ be such that $\Int g \circ f$, $\Int g' \circ f$ preserves $B$ and $T$, $B'$ and $T'$ respectively.

Take $h\in G$ so that $hT'h^{-1}=T$ and $hB'h^{-1}=B$.
Then
\[
\xymatrix{
G \ar[r]^{\Int g' \circ f} \ar[d]_{\Int h}& G \ar[d]^{\Int h}\\
G \ar[r]_{\Int( \alpha)\circ f}& G
}
\]
commutes where $ \alpha=hg'f(h^{-1})$.
In particular, we have
$\Int \alpha \circ f$ also preserves $B$ and $T$, and
$\delta_{(\Int g' \circ f)|_{T'}}=\delta_{(\Int \alpha \circ f)|_{T}}$.
Note that $g \alpha^{-1}\in N_{G}(B)\cap N_{G}(T)=B\cap N_{G}(T)=N_{B}(T)=Z_{B}(T)$.
Therefore, $(\Int \alpha \circ f)|_{T}=(\Int g \circ f)|_{T}$.
Thus, we get $\delta_{(\Int g' \circ f)|_{T'}}=\delta_{(\Int g \circ f)|_{T}}$.

Combining Step 3 and Step4, we get the theorem.

\end{proof}

\begin{prop}\label{split of endom of solv grp}
Let $G=T\ltimes U$ be a connected solvable group and $f \colon G \longrightarrow G$
a surjective group endomorphism.
If $f$ has a Zariski dense orbit, then $f$ preserves $T$ and $U$, $G=T\times U$,
and $f=f|_{T}\times f|_{U}$.
\end{prop}
\begin{proof}
Since $f$ is a group homomorphism, it induces $f_{U} \colon U \longrightarrow U$.
Thus, $f$ induces $f_{G/U} \colon G/U \longrightarrow G/U \simeq T$.
Since $f$ has a Zariski dense orbit, $f_{G/U}$ also has a Zariski dense orbit.
Take $g\in G$ such that $gf(T)g^{-1}=T$ and set $f'=\Int g \circ f$.
Then $f'$ preserves $T$ and $U$.
Moreover, since $f'$ and $f$ induces the same homomorphism on $G/U$ which we write $f'_{G/U}$,
it has a Zariski dense orbit.
Once we can prove $G=T \times U$, it is clear that $f$ is of the form $f|_{T} \times f|_{U}$.
Therefore, by replacing $f$ with $f'$, we may assume $f$ preserves $T$.

Now, $T$ acts on $U$ by conjugation and it corresponds to a group homomorphism $ \chi \colon T \longrightarrow \Aut U$.
We can check that the following diagram commutes:
\[
\xymatrix{
T \ar[r]^{f_{T}} \ar[d]_{ \chi} & T \ar[d]^{ \chi}\\
\Aut U \ar[r]_{\Int f_{U}}& \Aut U.
}
\]
Here $f_{T}$ and $f_{U}$ are endomorphism on $T$ and $U$ induced by $f$.
Note that since $U$ is connected unipotent, $f_{U}$ is an automorphism and $\Aut U$ is a linear algebraic group.
Let $T'= \chi(T)$.
Then $f_{U}$ is contained in the normalizer $N_{\Aut U}(T')$ of  $T'$.
Since $N_{\Aut U}(T')/Z_{\Aut U}(T')$ is finite (where $Z_{\Aut U}(T')$ is the centralizer) \cite[8.10 Corollary 2]{bor}, 
there exists a positive integer $n$ such that $(\Int f_{U})^{n}=\id$ on $T'$. 
Since $f_{T}$ has a Zariski dense orbit, $T'$ must be trivial and that means $G=T \times U$.
\end{proof}

\subsection{Proof of Theorem \ref{main: lin alg grp}}
Now we prove Theorem \ref{main: lin alg grp}.
In this section, the ground field is $ \overline{\mathbb Q}$. 

{\bf Convention}:
Let $f \colon X \longrightarrow X$ be a morphism where $X$ is a smooth quasi-projective variety defined over
$ \overline{\mathbb Q}$.
We say weak KSC is true for $f$ when the following holds:
for every point $x\in X( \overline{\mathbb Q})$, if the $f$-orbit of $x$ is Zariski dense in $X$,
then $ \alpha_{f}(x)$ exists and is equal to $\delta_{f}$.

\begin{prop}\label{red to solv grp}
If weak KSC is true for all surjective group endomorphisms of all connected solvable groups,
then weak KSC is true for all surjective group endomorphisms of connected linear algebraic groups.
\end{prop}
\begin{proof}
Let $f \colon G \longrightarrow G$ be a surjective group endomorphism of a linear algebraic group $G$.
Let $R_{u}G$ be the unipotent radical.
Then $f$ induces an endomorphism $ \overline{f} \colon G/R_{u}G \longrightarrow G/R_{u}G$.
Since any maximal torus of $G$ is isomorphic to a maximal torus of $G/R_{u}G$, 
Theorem \ref{dynamical deg of endom of alg group} implies $\delta_{f}=\delta_{ \overline{f}}$.
Thus, we may assume $G$ is reductive.
(Note that $ \underline{\alpha}_{f}(x)\geq \underline{\alpha}_{ \overline{f}}(xR_{u}G)$ for all $x\in G$.)

By \cite[Theorem 7.2]{stein}, there exists a Borel subgroup $B=T\ltimes U \subset G$ such that $f(B)=B$
where $T$ is a maximal torus and $U$ is the maximal unipotent subgroup.
Let $U^{-}$ be the opposite unipotent group of $U$.
Then we get
\[
\xymatrix@C=46pt{
U^{-}\times B \ar[d] \ar[r]^{f|_{U^{-}}\times f|_{B}} & U^{-}\times B\ar[d] \\
G \ar[r]_{f} & G
}
\] 
where the vertical arrows are open immersions.
Let $x\in G(\QQ)$ be a point with Zariski dense $f$-orbit.
By replacing $x$ by $f^{k}(x)$ for some $k$, we may assume $x\in (U^{-}\times B)(\QQ)$.
Write $x=(y,z)\in U^{-}(\QQ)\times B(\QQ)$.
Since $\delta_{f|_{U^{-}}}=1$, we  have $ \alpha_{f|_{U^{-}}}(y)=1$.
Since $x$ has Zariski dense $f$-orbit, $z$ has Zariski dense $f|_{B}$-orbit.
Thus, we get $ \alpha_{f}(x)=\max\{ \alpha_{f|_{U^{-}}}(y), \alpha_{f|_{B}}(z)\}= \alpha_{f|_{B}}(z)=\delta_{f|_{B}}$.
Here, we use the assumption in the last equality (and the existence of $ \alpha_{f|_{B}}(z)$).

\end{proof}

\begin{lem}\label{ksc for solv grp}
Let $G=T \ltimes U$ be a connected solvable group where $T$ is a maximal torus and $U$ is the maximal unipotent subgroup.
Let $f \colon G \longrightarrow G$ be a surjective group endomorphism.
If $f(T)=T$, then weak KSC is true for $f$.
\end{lem}
\begin{proof}
Note that $f=f|_{T}\times f|_{U}$.
If a point $(x,y)\in T\ltimes U=G$ has Zariski dense $f$-orbit, then $x$ also has Zariski dense $f|_{T}$-orbit.
By KSC for group endomorphisms on algebraic tori \cite{sil}, 
we have $ \alpha_{f}(x,y)=\max\{ \alpha_{f|_{T}}(x), \alpha_{f|_{U}}(y)\}= \alpha_{f|_{T}}(x)=\delta_{f|_{T}}=\delta_{f}$. 
(In the last equality, we use Theorem \ref{dynamical deg of endom of alg group}.)
\end{proof}

\begin{proof}[Proof of Theorem \ref{main: lin alg grp}]
The theorem follows from
Proposition \ref{red to solv grp}, \ref{split of endom of solv grp}, and Lemma \ref{ksc for solv grp}.
\end{proof}

\end{document}